\documentclass[12pt]{amsart}
\usepackage{amscd,amsmath,amsthm,amssymb}
\usepackage{amsfonts,amssymb,amscd,amsmath,enumerate,verbatim}
\usepackage[left]{lineno}
\usepackage{pstricks, pst-plot,pst-3d}
\usepackage{tikz}
\usepackage{epsfig}
\newpsstyle{fatline}{linewidth=1.5pt}
\definecolor{verylight}{gray}{0.97}
\definecolor{light}{gray}{0.9}
\definecolor{medium}{gray}{0.85}
\definecolor{dark}{gray}{0.6}
\def\NZQ{\mathbb}               

\def\ZZ{{\NZQ Z}}
\def\RR{{\NZQ R}}

\def\Cc{{\mathcal C}}

\def\opn#1#2{\def#1{\operatorname{#2}}} 
\opn\chara{char} \opn\length{\ell} \opn\pd{pd} \opn\rk{rk}
\opn\projdim{proj\,dim} \opn\injdim{inj\,dim} \opn\rank{rank}
\opn\depth{depth} \opn\grade{grade} \opn\height{height}
\opn\reg{reg}
\opn\embdim{emb\,dim} \opn\codim{codim}
\opn\Cl{Cl}

\opn\Tr{Tr} \opn\bigrank{big\,rank}
\opn\superheight{superheight}\opn\lcm{lcm}
\opn\trdeg{tr\,deg}
	\opn\reg{reg} \opn\lreg{lreg} \opn\ini{in} \opn\lpd{lpd}
	\opn\size{size} \opn\sdepth{sdepth}
	\opn\link{link}
    \opn\fdepth{fdepth}\opn\lex{lex}
	\opn\tr{tr}\opn\del{del}
	\opn\type{type}
	\opn\gap{gap}
	\opn\arithdeg{arith-deg}
	\opn\revlex{revlex}
	\opn\div{div} \opn\Div{Div} \opn\cl{cl} \opn\Cl{Cl}
	\opn\Spec{Spec} \opn\Supp{Supp} \opn\supp{supp} \opn\Sing{Sing}
	\opn\Ass{Ass} \opn\Min{Min}\opn\Mon{Mon}
	\opn\Ann{Ann} \opn\Rad{Rad} \opn\Soc{Soc}
	\opn\Im{Im} \opn\Ker{Ker} \opn\Coker{Coker} \opn\Am{Am}
	\opn\Hom{Hom} \opn\Tor{Tor} \opn\Ext{Ext} \opn\End{End}
	\opn\Aut{Aut} \opn\id{id}
	
	\opn\nat{nat}
	\opn\pff{pf}
	\opn\Pf{Pf} \opn\GL{GL} \opn\SL{SL} \opn\mod{mod} \opn\ord{ord}
	\opn\Gin{Gin} \opn\Hilb{Hilb}\opn\sort{sort}
	\opn\PF{PF}\opn\Ap{Ap}
	\opn\mult{mult}
	\opn\bight{bight}
    \opn\adj{adj}
	\opn\div{div}
	\opn\Div{Div}
	\opn\aff{aff}
	\opn\relint{relint} \opn\st{st}
	\opn\lk{lk} \opn\cn{cn} \opn\core{core} \opn\vol{vol}  \opn\inp{inp} \opn\nilpot{nilpot}
	\opn\link{link} \opn\star{star}\opn\lex{lex}\opn\set{set}
	\opn\width{wd}
	\opn\Fr{F}
	\opn\QF{QF}
	\opn\G{G}
	\opn\type{type}\opn\res{res}
	\opn\conv{conv}
	\opn\Deg{Deg}
	\opn\Sym{Sym}
	\opn\Con{Con}
	\opn\gr{gr}
	
	\def\pot#1#2{#1[\kern-0.28ex[#2]\kern-0.28ex]}

	\opn\dirlim{\underrightarrow{\lim}}
	\opn\inivlim{\underleftarrow{\lim}}
	\def\Implies{\ifmmode\Longrightarrow \else
		\unskip${}\Longrightarrow{}$\ignorespaces\fi}
	\def\implies{\ifmmode\Rightarrow \else
		\unskip${}\Rightarrow{}$\ignorespaces\fi}
	\def\iff{\ifmmode\Longleftrightarrow \else
		\unskip${}\Longleftrightarrow{}$\ignorespaces\fi}

	\let\:=\colon
	\newtheorem{Theorem}{Theorem}[section]
	\newtheorem{Lemma}[Theorem]{Lemma}
	\newtheorem{Corollary}[Theorem]{Corollary}

	\newtheorem{Example}[Theorem]{Example}

	\newtheorem{Conjecture}[Theorem]{Conjecture}
	
	\newtheorem{Question}[Theorem]{Question}
	\let\epsilon\varepsilon
	\let\kappa=\varkappa
	\def\qed{\ifhmode\textqed\fi
		\ifmmode\ifinner\quad\qedsymbol\else\dispqed\fi\fi}
	\def\textqed{\unskip\nobreak\penalty50
		\hskip2em\hbox{}\nobreak\hfil\qedsymbol
		\parfillskip=0pt \finalhyphendemerits=0}
	\def\dispqed{\rlap{\qquad\qedsymbol}}
	
	\opn\dis{dis}
	\def\pnt{{\raise0.5mm\hbox{\large\bf.}}}
	
	\opn\Lex{Lex}

\begin{document}
\title[Regularity and depth of binomial ideals]{Regularity and depth of binomial ideals arising from combinatorics}

\author[T.~Hibi]{Takayuki Hibi}
\author[S.~A.~ Seyed Fakhari]{Seyed Amin Seyed Fakhari}

\address{(Takayuki Hibi) Department of Pure and Applied Mathematics, Graduate School of Information Science and Technology, Osaka University, Suita, Osaka 565--0871, Japan}
\email{hibi@math.sci.osaka-u.ac.jp}
\address{(Seyed Amin Seyed Fakhari) Departamento de Matem\'aticas, Universidad de los Andes, Bogot\'a, Colombia}
\email{s.seyedfakhari@uniandes.edu.co}



\begin{abstract}
Regularity and depth of binomial ideals generated by adjacent $2$-minors together with those arising from finite lattices are studied.

\end{abstract}	
\maketitle
\thispagestyle{empty}

\section*{Introduction}
Combinatorics on the classical theory of lattices and graphs has been creating fascinating research projects in commutative algebra.  In particular, ideals generated by quadratic binomials \cite{HH, HHHKR, Hibi, Q} have been studied for quarter century by a huge number of research papers and, together with monomial ideals, belong to current trends on combinatorics and commutative algebra.  The purpose of the present paper is to study the regularity and the depth of binomial ideals arising from combinatorics.  More precisely, the  regularity and the depth of binomial ideals generated by adjacent $2$-minors \cite{HH} together with those arising from finite lattices \cite{DEH, EH, Hibi} is investigated.

\section{Binomial ideals generated by adjacent $2$-minors}
Let $X = (x_{ij})_{i=1,\ldots,m \atop  j=1,\ldots,n}$  be an $m \times n$ matrix of indeterminates.  An \emph{adjacent $2$--minor} of $X$ is a binomial of the form $ x_{i,j}x_{i+1,j+1} - x_{i+1,j}x_{i,j+1}$.  Let $K[X]=K[(x_{ij})_{i=1,\ldots,m \atop  j=1,\ldots,n}]$ denote the polynomial ring in $mn$ variables over a field $K$.  A {\em cell} is a square $C \subset \RR^2$ whose vertices are $(i,j),(i,j+1), (i+1,j), (i+1,j+1)$, where $i,j\in \ZZ$ with $1 \leq i<m, 1 \leq j<n$.  We associate such a cell $C$ with the binomial $f_{C} = x_{i,j}x_{i+1,j+1} - x_{i+1,j}x_{i,j+1}$.  Let $V(C)$ (resp. $E(C)$) denote the set of vertices (resp. edges) of a cell $C$.   Let $\Cc$ be a collection of cells and $S=S_\Cc=K[x_{v}: v \in V(\Cc)]$ the polynomial ring in $|V(\Cc)|$ variables over $K$, where $V(\Cc) = \bigcup_{C \in \Cc} V(C)$.  Let $E(\Cc) = \bigcup_{C \in \Cc} E(C)$.  Given a collection $\Cc$ of cells, we introduce the binomial ideal 
\[
I_{\rm adj}(\Cc) = (f_C : C \in \Cc) \subset S. 
\] 

A collection $\Cc$ of cells is  {\em row convex} if the following condition is satisfied: For any pair of cells  $C,C'$ belonging to $\Cc$ with $V(C)=\{(i,j),(i,j+1), (i+1,j), (i+1,j+1)\}$ and $V(C') =\{(i',j),(i',j+1), (i'+1,j), (i'+1,j+1)\}$, with $i < i'$, the cell $C''$ with $V(C'') =\{(i'',j),(i'',j+1), (i''+1,j), (i''+1,j+1)\}$ belongs to $\Cc$, for each integer $i''$ with $i< i''<i'$.  A {\em column convex} collection of cells is defined similarly. A collection $\Cc$ of cells is {\em convex} if it is both row convex and column convex.  A collection $\Cc$ of cells is {\em connected} if for $C \in \Cc$ and $C' \in \Cc$, there is a sequence $C=C_0, C_1, \ldots, C_{\ell-1}, C_\ell = C'$ of cells of $\Cc$, called a {\em path} between $C$ and $C'$, for which $|E(C_i) \cap E(C_{i+1})| =1$ for each $0 \leq i < \ell$.

A collection of cells of Figure $1$ (resp. Figure $2$) is called a {\em square tetromino} (resp. an {\em $X$-pentomino}). The labeling of the vertices in Figure 2, will be used in the proof of Lemma \ref{Xpent}. It is shown \cite[Theorems 3.1 and 3.5]{HNQS} that if $\Cc$ is convex, then $I_{\rm adj}(\Cc)$ is a complete intersection if and only if $\Cc$ contains neither a square tetromino nor an $X$-pentomino.  It follows from the computation by using Koszul complex that if $I_{\rm adj}(\Cc)$ is a complete intersection, then ${\rm reg}(I_{\rm adj})(\mathcal{C}) = |\Cc| + 1$.

\begin{figure}[h]
    \centering
    \includegraphics[scale=0.85]{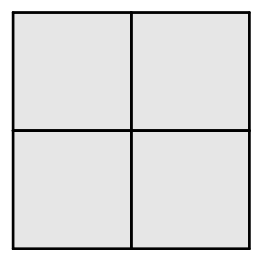}
    \caption{A square tetromino.}
    \label{}
\end{figure}

\begin{figure}[h]
    \centering
    \includegraphics[scale=0.3]{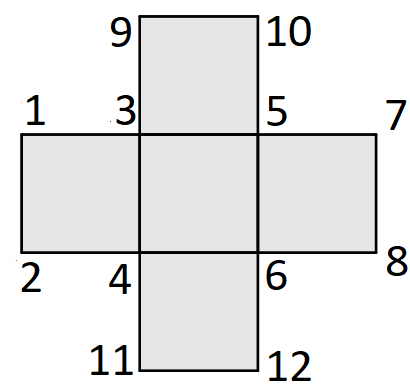}
    \caption{An $X$-pentomino.}
    \label{}
\end{figure}

\begin{Example}
{\em Let $I_{\rm adj}(\Cc)$ is the ideal generated by adjacent $2$-minors of a $3 \times 4$ matrix of indeterminates.  Then ${\rm reg}(I_{\rm adj}(\mathcal{C}) = 5$ and ${\rm depth}(S/I_{\rm adj}(\mathcal{C}))=4$. 

}    
\end{Example}

\begin{Lemma} \label{Xpent}
Let $\mathcal{C}$ denote the $X$-pentomino. Then $${\rm reg}(I_{\rm adj}(\mathcal{C}))=5, \quad {\rm depth}(I_{\rm adj}(\mathcal{C}))=8.$$
\end{Lemma}

\begin{proof}
First we prove ${\rm reg}(I_{\rm adj}(\mathcal{C}))=5$.  Let, say, $V(\mathcal{C})=[12]=\{1,\ldots,12\}$ and suppose that the vertices of $\Cc$ are labeled as in Figure 2. To simplify the notation, we denote the variables of $S$ with $x_1, \ldots, x_{12}$ (instead of using double-indexed variables). Then$$I_{\rm adj}(\mathcal{C})=(x_1x_4-x_2x_3, x_3x_6-x_4x_5,x_5x_8-x_6x_7, x_5x_9-x_3x_{10},x_4x_{12}-x_6x_{11}).$$Set $I:=I_{\rm adj}(\mathcal{C})$ and $S':=K[x_1, \ldots, x_{10}]$.  Since
\[
I':=I\cap S'=x_1x_4-x_2x_3, x_3x_6-x_4x_5,x_5x_8-x_6x_7, x_5x_9-x_3x_{10})
\]
is a complete intersection \cite[Theorems 3.1 and 3.5]{HNQS}, it follows that $I'$ is resolved by Koszul complex. This yields that ${\rm reg}(I')=5$. On the other hand, $S'/I'\subseteq S/I$ is an algebra retract. It then follows from \cite[Corollary 2.5]{OHH} that ${\rm reg}(I')\leq {\rm reg}(I)$. In particular, ${\rm reg}(I)\geq 5$.  We show that ${\rm reg}(I)\leq 5$. 

Let $<_{\rm rev}$ denote the reverse lexicographic order on $S$ induced by the ordering $x_1> x_2> \cdots > x_{12}$ of variables.  One easily sees that
\begin{eqnarray*}
&\{ x_1x_4-x_2x_3, x_3x_6-x_4x_5,x_5x_8-x_6x_7, x_5x_9-x_3x_{10},x_4x_{12}-x_6x_{11},&\\ & x_5x_8x_{11}-x_4x_7x_{12}, x_3x_6x_9-x_3x_4x_{10}, x_4^2x_7x_{12}-x_3x_4x_8x_{12},&\\ & x_3x_8x_{10}x_{11}-x_4x_7x_9x_{12}, x_3x_4x_{10}x_{11}-x_3x_4x_9x_{12}, x_3x_4x_7x_{10}-x_3^2x_8x_{10},&\\ & x_1x_4x_6x_9-x_1x_4^2x_{10},x_1x_4x_8x_{10}x_{11}-x_2x_4x_7x_9x_{12},x_1x_4^2x_{10}x_{11}-x_1x_4^2x_9x_{12},&\\ & x_1x_4^2x_7x_{10}-x_1x_3x_4x_8x_{10}\}&     
\end{eqnarray*}
is a Gr\"obner basis of $I$ with respect to $<_{\rm rev}$. Set $J:={\rm in}_{\leq}(I)$. One has 
\begin{eqnarray*}
J=(x_2x_3, x_4x_5, x_6x_7, x_5x_9, x_6x_{11},  x_5x_8x_{11}, x_3x_6x_9, x_4^2x_7x_{12}, x_3x_8x_{10}x_{11},\\  x_3x_4x_{10}x_{11}, x_3x_4x_7x_{10}, x_1x_4x_6x_9,x_1x_4x_8x_{10}x_{11},x_1x_4^2x_{10}x_{11},  x_1x_4^2x_7x_{10}).     
\end{eqnarray*}
By virtue of \cite[Theorem 3.3.4]{HHgtm260}, it is enough to show that ${\rm reg}(J)\leq 5$.

Considering the short exact sequence$$0\longrightarrow \frac{S}{(J:x_4)}(-1)\longrightarrow \frac{S}{J}\longrightarrow \frac{S}{(J,x_4)}\longrightarrow 0,$$
we deuce that
\[
\begin{array}{rl}
{\rm reg}(J)\leq\max\{{\rm reg}(J:x_4)+1,{\rm reg}(J,x_4)\}.
\end{array} \tag{1} \label{1}
\]
Note that $$(J,x_4)=(x_4, x_2x_3,x_6x_7,x_5x_9,x_6x_{11},x_5x_8x_{11},x_3x_6x_9,x_3x_8x_{10}x_{11}).$$
Set $J_1=(J,x_4)$.  Furthermore, considering the short exact sequence$$0\longrightarrow \frac{S}{(J_1:x_{11})}(-1)\longrightarrow \frac{S}{J_1}\longrightarrow \frac{S}{(J_1,x_{11})}\longrightarrow 0,$$
we conclude that
\[
\begin{array}{rl}
{\rm reg}(J_1)\leq\max\{{\rm reg}(J_1:x_{11})+1,{\rm reg}(J,x_{11})\}.
\end{array} \tag{2} \label{2}
\] 
Observe that$$(J_1,x_{11})=(x_4,x_{11}, x_2x_3,x_6x_7,x_5x_9,x_3x_6x_9).$$It follows from Taylor resolution that ${\rm reg}(J_1,x_{11})\leq 4$.  Furthermore, $$(J_1:x_{11})=(x_4, x_6,x_2x_3,x_5x_9,x_5x_8,x_3x_8x_{10})$$is the sum of two ideals $(x_4,x_6,x_2x_3,x_3x_8x_{10})$ and $(x_5x_9,x_5x_8)$ which have linear quotients.  Thus, by using \cite[Corollary 3.2]{H}, one has ${\rm reg}(J_1:x_{11})\leq 4$.  Hence, the inequality (\ref{2}) yields that ${\rm reg}(J_1)\leq 5$.  

Recall that $J_1=(J,x_4)$.  Seeing (\ref{1}), it is enough to show that ${\reg}(J:x_4)\leq 4$. Set $J_2:=(J:x_4)$.  One has
\begin{align*}
J_2=( x_5,x_2x_3, x_6x_7, x_6x_{11}, x_3x_6x_9, x_4x_7x_{12}, x_3x_{10}x_{11}, x_3x_7x_{10},\\  x_1x_6x_9,x_1x_8x_{10}x_{11},x_1x_4x_{10}x_{11},  x_1x_4x_7x_{10}).
\end{align*}
As for $J_1$, considering the short exact sequence$$0\longrightarrow \frac{S}{(J_2:x_3)}(-1)\longrightarrow \frac{S}{J_2}\longrightarrow \frac{S}{(J_2,x_3)}\longrightarrow 0,$$
we conclude that
\[
\begin{array}{rl}
{\rm reg}(J_2)\leq\max\{{\rm reg}(J_2:x_3)+1,{\rm reg}(J,x_3)\}.
\end{array} \tag{3} \label{3}
\] 
Since $$(J_2,x_3)=(x_3, x_5, x_6x_7, x_6x_{11}, x_4x_7x_{12}, x_1x_6x_9,x_1x_8x_{10}x_{11},x_1x_4x_{10}x_{11}, x_1x_4x_7x_{10})$$has linear quotients, one has ${\rm reg}(J_2,x_3)=4$.  Moreover, as$$(J_2:x_3)=(x_2, x_5, x_6x_7, x_6x_{11}, x_6x_9, x_4x_7x_{12}, x_{10}x_{11}, x_7x_{10})$$has linear quotients, one has ${\rm reg}(J_2 : x_3)=3$.  Hence, inequality (\ref{3}) yields that ${\rm reg}(J_2)\leq 4$.  This completes the proof of ${\rm reg}(I_{\rm adj}(\mathcal{C}))=5$.

Second, we prove ${\rm depth}(I_{\rm adj}(\mathcal{C}))=8$.  Since the ideal generated by all inner 2-minors of $\mathcal{C}$ is a minimal prime of $I$ of height $5$ (\cite[Theorem 8.11]{HHO}), one has ${\rm depth}(I)\leq 8$. The proof of the reverse inequality is similar to the argument for the regularity. So, we omit the details. The only difference is that when a monomial ideal has linear quotients, its depth can be computed by \cite[Corollary 8.2.2]{HHgtm260}.
\end{proof}

Let $\Cc$ be a collection of cells.  A vertex $v$ of $\Cc$ is called {\em free} if $v$ belongs to exactly one cell of $\Cc$.  We say that a cell $C \in \Cc$ {\em is adjacent to} $C' \in \Cc$ if $|E(C) \cap E(C')|= 1$.  

\begin{Lemma} \label{associated}
Let $\mathcal{C}$ denote the $X$-pentomino. If $i,j$ are free vertices of $\mathcal{C}$ which belong to a cell $C$ of $\mathcal{C}$, then for any associated prime $\mathfrak{p}\in\Ass{S/I_{\rm adj}(\mathcal{C})}$, one has either $x_i\notin\mathfrak{p}$ or $x_j\notin\mathfrak{p}$.    
\end{Lemma}

\begin{proof}
Since ${\rm depth}(S/I_{\rm adj}(\mathcal{C}))=7$ (Lemma \ref{Xpent}), one has $\height{\mathfrak{p}}\leq 5$. There is a prime ideal $\mathfrak{q}\in \Min{I_{\rm adj}(\mathcal{C})}$ with $\mathfrak{q}\subseteq \mathfrak{p}$. Since the structure of minimal primes of $I_{\rm adj}(\mathcal{C})$ is known \cite[Theorem 8.11]{HHO}, one has $\height{(\mathfrak{q}+(x_i,x_j))}\geq 6$. Since $\mathfrak{q}\subseteq \mathfrak{p}$ and $\height{\mathfrak{p}}\leq 5$, we deduce that either $x_i\notin\mathfrak{p}$ or $x_j\notin\mathfrak{p}$. \, \, \, \, \, \,   \end{proof}

We introduce a generalized $X$-pentomino.  A collection $\Cc$ of cells is called a {\em row of cells} if $V(\Cc)=\{(i,j), (i,j+1) : p \leq i \leq q\}$ and is called a {\em column of cells} if $V(\Cc)=\{(i,j), (i+1,j) : p \leq j \leq q\}$, where $i, j, p, q \in \ZZ_{>0}$ with $p < q$.  A {\em generalized $X$-pentomino} is a collection $\Cc$ of cells of the form $\Cc = \Cc' \cup \Cc''$, where $\Cc'$ is a row of cells and $\Cc''$ is a column of cells, for which $\Cc$ contains an $X$-pentomino (see e.g.,Figure 3).  The cell $C$ with $\{C\} = \Cc' \cap \Cc''$ is called the {\em central cell} of the generalized $X$-pentomino $\Cc$.  

It follows that a connected and convex collection of cells containing an $X$-pentomino and no square tetromino is a generalized $X$-pentomino.

\begin{figure}
\centering
\begin{tikzpicture}[scale=0.6]
\filldraw [fill=black!15!white,draw=black] (0,0) rectangle (1,1);
\filldraw [fill=black!15!white,draw=black] (1,0) rectangle (2,1);
\filldraw [fill=black!15!white,draw=black] (2,0) rectangle (3,1);
\filldraw [fill=black!15!white,draw=black] (3,0) rectangle (4,1);
\filldraw [fill=black!15!white,draw=black] (4,0) rectangle (5,1);
\filldraw [fill=black!15!white,draw=black] (5,0) rectangle (6,1);
\filldraw [fill=black!15!white,draw=black] (6,0) rectangle (7,1);
\filldraw [fill=black!15!white,draw=black] (7,0) rectangle (8,1);
\filldraw [fill=black!15!white,draw=black] (3,1) rectangle (4,2);
\filldraw [fill=black!15!white,draw=black] (3,2) rectangle (4,3);
\filldraw [fill=black!15!white,draw=black] (3,-1) rectangle (4,0);
\filldraw [fill=black!15!white,draw=black] (3,-2) rectangle (4,-1);
\filldraw [fill=black!15!white,draw=black] (3,-3) rectangle (4,-2);
\end{tikzpicture}
\caption{A generalized $X$-pentomino.}
\end{figure}
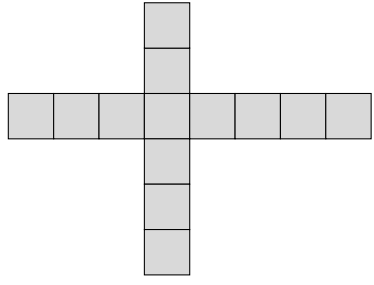

\begin{Theorem} \label{depth}
Let $\mathcal{C}$ be a generalized $X$-pentomino and $D$ the central cell of $\mathcal{C}$. Then the followings hold:
\begin{itemize}
\item [(i)] ${\rm depth}(S/I_{\rm adj}(\mathcal{C}))=|V(\mathcal{C})|-|\mathcal{C}|$.     
\item [(ii)] Suppose that $C$ is a cell of $\mathcal{C}$ which is adjacent to $D$ and has two free vertices $i,j$. Then for any associated prime $\mathfrak{p}\in\Ass{S/I_{\rm adj}(\mathcal{C})}$, one has either $x_i\notin\mathfrak{p}$ or $x_j\notin\mathfrak{p}$.
\end{itemize}     
\end{Theorem}

\begin{proof}
We prove both parts simultaneously by induction on $|\mathcal{C}|$. It follows from \cite[Theorem 8.11]{HHO} that the ideal generated by all inner $2$-minors \cite[p.~248]{HHO} of $\mathcal{C}$ is a minimal prime of $I_{\rm adj}(\mathcal{C})$ of height $|\mathcal{C}|$. Hence, ${\rm depth}(S/I_{\rm adj}(\mathcal{C}))\leq |V(\mathcal{C})|-|\mathcal{C}|$. So, in order to prove (i), we only need to prove ${\rm depth}(S/I_{\rm adj}(\mathcal{C}))\geq |V(\mathcal{C})|-|\mathcal{C}|$.

If $\mathcal{C}$ is the $X$-pentomino, then (i) and (ii) follow from Lemmata \ref{Xpent} and \ref{associated}. So, suppose that $|\mathcal{C}|\geq 6$. Recall that $D$ denotes the central cell of $\mathcal{C}$. Since $|\mathcal{C}|\geq 6$, without loss of generality, we may assume that there is a row of distinct cells $\{D,D',D''\}$ in $\mathcal{C}$, for which $D'$ is adjacent to $D$ and $D''$ is adjacent to $D'$. To simplify the notation, let $V(D')\cap V(D'')=\{1,2\}$, $V(D'')\setminus V(D')=\{3,4\}$ and the adjacent $2$-minor corresponding to $D''$ is $x_1x_4-x_2x_3$.  Set $\mathcal{C}':=\mathcal{C}\setminus\{D''\}$, $I:=I_{\rm adj}(\mathcal{C})$ and $I':=I_{\rm adj}(\mathcal{C}')S$. Note that $\mathcal{C}'$ is the disjoint union of a generalized $X$-pentomino $\mathcal{C}''$ and a (possibly empty) row of cells, say, $\mathcal{P}$.  Since the ideal of adjacent $2$-minors of a row of cells is a complete intersection, it follows from the induction hypothesis that$${\rm depth}(S/I')=|V(\mathcal{C})|-|\mathcal{C}'|=|V(\mathcal{C})|-|\mathcal{C}|+1.$$

\medskip
\noindent
{\bf Claim.} $x_1x_4-x_2x_3$ is regular on $S/I'$.

\medskip
\noindent
({\it Proof of the claim.}) We prove that for any prime ideal $\mathfrak{p}'\in\Ass{S/I'}$, one has $x_1x_4-x_2x_3\notin \mathfrak{p}'$. By contradiction, assume that $x_1x_4-x_2x_3\in \mathfrak{p}'$. We define a bigrading on $S$ as follows. For each variable $x_k$ we define $\deg(x_k)=(r,t)$ if the vertex $k$ belongs to row $r$ and column $t$ (of a fixed $m \times n$ matrix $X = (x_{ij})_{i=1,\ldots,m \atop  j=1,\ldots,n}$).  Then $I'$ is a homogeneous ideal with respect to this bigrading. Since $\mathfrak{p}'\in\Ass{S/I'}$, it follows that $\mathfrak{p}'$ must be homogeneous. In particular, if a linear form $\ell$ belongs to $\mathfrak{p}'$, then any variable appearing in $\ell$ must belong to $\mathfrak{p}'$. On the other hand, since the ideal of adjacent $2$-minors of a row (or a column) of cells is a complete intersection, one has $\Ass{S/I_{\rm adj}(\mathcal{P})}=\Min{I_{\rm adj}(\mathcal{P})}$. Moreover, it follows from \cite[Theorem 2.5]{HNTT} that $\mathfrak{p}'=\mathfrak{q}_1+\mathfrak{q}_2$, where $\mathfrak{q}_1\in\Ass{S/I_{\rm adj}(\mathcal{C}'')}$ and $\mathfrak{q}_2\in\Ass{S/I_{\rm adj}(\mathcal{P})}=\Min{I_{\rm adj}(\mathcal{P})}$. It follows from the structure of minimal primes of $I_{\rm adj}(\mathcal{P})$ \cite[Theorem 8.11]{HHO}, that $x_3, x_4\notin \mathfrak{q}_2$. We deduce from $x_1x_4-x_2x_3\in \mathfrak{p}'=\mathfrak{q}_1+\mathfrak{q}_2$ that there are linear forms $\ell,\ell'\in \mathfrak{q}_1$ involving $x_1,x_2$, respectively. Hence, our previous observation implies that $x_1,x_2\in \mathfrak{q}_1$, contradicting our induction hypothesis on (ii). This proves the claim. 

\medskip

Now, the short exact sequence$$0\longrightarrow \frac{S}{(I':x_1x_4-x_2x_3)}\longrightarrow \frac{S}{I'}\longrightarrow \frac{S}{I}\longrightarrow 0$$
yields that
\[
\begin{array}{rl}
{\rm depth}(S/I)\geq\min\{{\rm depth}(S/(I':x_1x_4-x_2x_3))-1,{\rm depth}(S/I')\}.
\end{array} \tag{4} \label{4}
\]  
The claim guarantees that $(I': x_1x_4-x_2x_3)=I'$. It then follows from  (\ref{4}) and ${\rm depth}(S/I')=|V(\mathcal{C})-|C|+1$ that ${\rm depth}(S/I_{\rm adj}(\mathcal{C}))\geq |V(\mathcal{C})|-|\mathcal{C}|$, proving (i).

We prove (ii). Suppose that $C$ is a cell of $\mathcal{C}$ which is adjacent to $D$ and has two free vertices $i,j$.  We deduce from (i) that $\height{\mathfrak{p}}\leq |\mathcal{C}|$. There is a minimal prime ideal $\mathfrak{q}$ of $I$ 
with $\mathfrak{q}\subseteq \mathfrak{p}$. It follows from the structure of minimal prime ideals \cite[Theorem 8.11]{HHO}, that $\height{(\mathfrak{q}+(x_i,x_j))}\geq |\mathcal{C}|+1$. Since $\mathfrak{q}\subseteq \mathfrak{p}$ and $\height{\mathfrak{p}}\leq |\mathcal{C}|$, we conclude that either $x_1\notin\mathfrak{p}$ or $x_2\notin\mathfrak{p}$, proving (ii). 
\end{proof}

\begin{Corollary} \label{pd}
Let $\mathcal{C}$ be a generalized $X$-pentomino. Then$${\rm pd}(S/I_{\rm adj}(\mathcal{C}))=\bight{I_{\rm adj}(\mathcal{C})}=|\mathcal{C}|.$$          
 \end{Corollary}

\begin{proof}
The equality ${\rm pd}(S/I_{\rm adj}(\mathcal{C}))=|\mathcal{C}|$ follows from Theorem \ref{depth} together with Auslander-Buchsbaum formula. Furthermore, a general fact on projective dimension and on big height says that ${\rm pd}(S/I_{\rm adj}(\mathcal{C}))\geq \bight{I_{\rm adj}(\mathcal{C})}$. Since the ideal generated by all inner $2$-minors of $\mathcal{C}$ is a minimal prime ideal of $I_{\rm adj}(\mathcal{C})$ of height $|\mathcal{C}|$, the other inequality ${\rm pd}(S/I_{\rm adj}(\mathcal{C}))\leq \bight{I_{\rm adj}(\mathcal{C})}$ follows. 
\, \, \, \, \, \, \, \, \, \, \, \, \, \, \, \, \, \, \, \, \, 
\end{proof}

\begin{Theorem} \label{reg}
Let $\mathcal{C}$ be a generalized $X$-pentomino. Then ${\rm reg}(I_{\rm adj}(\mathcal{C}))=|\mathcal{C}|$.    
\end{Theorem}

\begin{proof}
We work with induction on $|\mathcal{C}|$. If $\mathcal{C}$ is $X$-pentomino, then the assertion follows from Lemma \ref{Xpent}. Let $|\mathcal{C}|\geq 6$. We follow the same notations as introduced in the proof of Theorem \ref{depth}.  We consider the short exact sequence$$0\longrightarrow \frac{S}{(I':x_1x_4-x_2x_3)}(-2)\longrightarrow \frac{S}{I'}\longrightarrow \frac{S}{I}\longrightarrow 0.$$
We know $(I':x_1x_4-x_2x_3)=I'$.  Applying the above short exact sequence yields that ${\rm reg}(S/I)={\rm reg}(S/I')+1$. Recall that $I'=I_{\rm adj}(\mathcal{C}')$ and $\mathcal{C}'$ is the disjoint union of a generalized $X$-pentomino $\mathcal{C}''$ and a (possibly empty) row of cells $\mathcal{P}$. Since the ideal of adjacent $2$-minors of $\mathcal{P}$ is a complete intersection, it follows from the induction hypothesis that ${\rm reg}(S/I')=|\mathcal{C}'|-1=|\mathcal{C}|-2.$ Thus ${\rm reg}(S/I)={\rm reg}(S/I')+1 = |\mathcal{C}|-1$, in other words, ${\rm reg}(I_{\rm adj}(\mathcal{C})=|\mathcal{C}|$, as desired.
\, \, \, \, \, \, \, \, \, \, \, \, \, \, \, \, \, \, \, \, \,\, \, \, \, 
\end{proof}

\begin{Example}
{\em
Let $\Cc, \Cc'$ and $\Cc''$ denote the collections of cells of Figure $4$ with $|\Cc|=7, |\Cc'|=8$ and $|\Cc''|=9$.  One has
\[
\reg(I_{\adj}(\Cc))=7, \, \, \, 
\reg(I_{\adj}(\Cc'))=7, \, \, \, 
\reg(I_{\adj}(\Cc''))=8
\]
and 
\[
\depth(S/I_{\adj}(\Cc))=9, \, \, \, 
\depth(S/I_{\adj}(\Cc'))=10, \, \, \, 
\depth(S/I_{\adj}(\Cc''))=11.
\]
}
\end{Example}

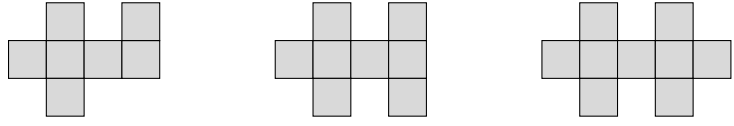
\begin{figure}
\centering
\begin{tikzpicture}[scale=0.5]
\filldraw [fill=black!15!white,draw=black] (0,0) rectangle (1,1);
\filldraw [fill=black!15!white,draw=black] (1,0) rectangle (2,1);
\filldraw [fill=black!15!white,draw=black] (2,0) rectangle (3,1);
\filldraw [fill=black!15!white,draw=black] (3,0) rectangle (4,1);
\filldraw [fill=black!15!white,draw=black] (1,-1) rectangle (2,0);
\filldraw [fill=black!15!white,draw=black] (1,1) rectangle (2,2);
\filldraw [fill=black!15!white,draw=black] (3,1) rectangle (4,2);
\end{tikzpicture} \ \ \ \ \ \ \ \ \ \ 
\begin{tikzpicture}[scale=0.5]
\filldraw [fill=black!15!white,draw=black] (0,0) rectangle (1,1);
\filldraw [fill=black!15!white,draw=black] (1,0) rectangle (2,1);
\filldraw [fill=black!15!white,draw=black] (2,0) rectangle (3,1);
\filldraw [fill=black!15!white,draw=black] (3,0) rectangle (4,1);
\filldraw [fill=black!15!white,draw=black] (1,-1) rectangle (2,0);
\filldraw [fill=black!15!white,draw=black] (1,1) rectangle (2,2);
\filldraw [fill=black!15!white,draw=black] (3,1) rectangle (4,2);
\filldraw [fill=black!15!white,draw=black] (3,-1) rectangle (4,0);
\end{tikzpicture} \ \ \ \ \ \ \ \ \ \ 
\begin{tikzpicture}[scale=0.5]
\filldraw [fill=black!15!white,draw=black] (0,0) rectangle (1,1);
\filldraw [fill=black!15!white,draw=black] (1,0) rectangle (2,1);
\filldraw [fill=black!15!white,draw=black] (2,0) rectangle (3,1);
\filldraw [fill=black!15!white,draw=black] (3,0) rectangle (4,1);
\filldraw [fill=black!15!white,draw=black] (1,-1) rectangle (2,0);
\filldraw [fill=black!15!white,draw=black] (1,1) rectangle (2,2);
\filldraw [fill=black!15!white,draw=black] (3,1) rectangle (4,2);
\filldraw [fill=black!15!white,draw=black] (3,-1) rectangle (4,0);
\filldraw [fill=black!15!white,draw=black] (4,0) rectangle (5,1);
\end{tikzpicture}
\caption{Collections of cells.}
\end{figure}

\section{Binomial ideals arising from finite lattices}
Let $L$ be a finite lattice and $K[L]=K[x_a : a \in L]$ the polynomial ring in $|L|$ variables over a field $K$.  Let $f_{a,b} = x_a x_b - x_{a \wedge b} x_{a \vee b}$, where $a,b \in L$.  The {\em join-meet ideal} of $L$ is the binomial ideal  
\[
I_L = (f_{a,b} : a, b \in L).
\]
In \cite{Hibi}, it is shown that $I_L$ is a prime ideal if and only if $L$ is a distributive lattice.  Furthermore, when $L$ is distributive, the quotient ring $K[L]/I_L$ is normal and Cohen--Macaulay.  On the other hand, $I_L$ has linear resolution if and only if $L = D_{2 \cdot 3^s}$, $s=1,2,\ldots$\,, the divisor lattice \cite[p.~157]{HHgtm260} of $2\cdot 3^s$ (\cite[Theorem 4.2]{EHH}).  If $L$ is a modular nondistributive lattice, then $\reg(I_L) \geq 4$ and $I_L$ is not linearly related (\cite[Theorem 2.2]{DEH}).   

\begin{Example}
\label{EX2.2}
{\em
Each lattice $L$ of Figure $5$ has $\reg(I_L) = 3$. 
}
\end{Example}

\begin{figure}
\centering
\begin{tikzpicture}[scale=0.8]
\coordinate (a) at (0,0) {};
\coordinate (b) at (-1,1.5) {};
\coordinate (c) at (0,3) {};
\coordinate (d) at (1,2) {};
\coordinate (e) at (1,1) {};
\draw(a)--(b)--(c)--(d)--(e)--cycle;
\draw[fill=white] (a) circle [radius=1.3mm];
\draw[fill=white] (b) circle [radius=1.3mm];
\draw[fill=white] (c) circle [radius=1.3mm];
\draw[fill=white] (d) circle [radius=1.3mm];
\draw[fill=white] (e) circle [radius=1.3mm];
\end{tikzpicture} \ \ \ \ \ \ \ \ \ \ \ \ \ \ \ \ \ 
\begin{tikzpicture}[scale=0.8]
\coordinate (a) at (0,0) {};
\coordinate (b) at (-1,1) {};
\coordinate (c) at (-1,2) {};
\coordinate (d) at (0,3) {};
\coordinate (e) at (1,2) {};
\coordinate (f) at (1,1) {};
\coordinate (g) at (0,2) {};
\coordinate (h) at (0,1) {};
\draw(a)--(b)--(c)--(d)--(e)--(f)--cycle;
\draw(b)--(e);
\draw(c)--(h)--(g)--(f);
\draw(d)--(g);
\draw(h)--(a);
\draw[fill=white] (a) circle [radius=1.3mm];
\draw[fill=white] (b) circle [radius=1.3mm];
\draw[fill=white] (c) circle [radius=1.3mm];
\draw[fill=white] (d) circle [radius=1.3mm];
\draw[fill=white] (e) circle [radius=1.3mm];
\draw[fill=white] (f) circle [radius=1.3mm];
\draw[fill=white] (g) circle [radius=1.3mm];
\draw[fill=white] (h) circle [radius=1.3mm];
\end{tikzpicture}
\caption{Nonmodular lattice and modular lattice}
\end{figure}

\begin{Example}
\label{aaaaa}
{\em 
    The lattice $L$ of Figure $6$ is a semimodular lattice with $\reg(I_L) =4$.  Since $\beta_{1,4}(I_L)=5$, it follows that $I_L$ is not linearly related.
    }
\end{Example}

\begin{figure}
\centering
\begin{tikzpicture}[scale=0.8]
\node[draw,shape=circle] (1) at (0,0) {};
\node[draw,shape=circle] (2) at (-1,1) {};
\node[draw,shape=circle] (3) at (-2,2) {};
\node[draw,shape=circle] (4) at (0,4) {};
\node[draw,shape=circle] (5) at (0,2) {};
\node[draw,shape=circle] (6) at (2,2) {};
\node[draw,shape=circle] (7) at (1,1) {};
\draw(1)--(2)--(3)--(4)--(6)--(7)--(1);
\draw(4)--(5);
\draw(2)--(5)--(7);
\end{tikzpicture}
\caption{A semimodular lattice.}
\end{figure}
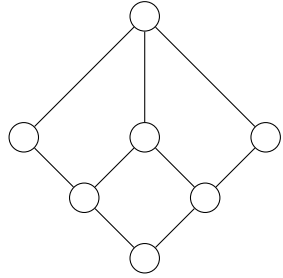

Is there a nonmodular lattice $L$ for which $I_L$ is linearly related?

\begin{Example}
\label{sixEX}
{\em
The join--meet ideal of each nonmodular lattice with at most six elements (\cite[Figure 3.5]{ECI}) is not linearly related.  Among them, $\reg(I_L) = 3$ if and only if $L$ is one of the lattices of Figure $7$. 
}
\end{Example}

\begin{figure}
\centering
\begin{tikzpicture}[scale=0.5]
\coordinate (a) at (0,0) {};
\coordinate (b) at (-1,1) {};
\coordinate (c) at (-1,2) {};
\coordinate (d) at (0,3) {};
\coordinate (e) at (1,1.5) {};
\draw(a)--(b)--(c)--(d)--(e)--cycle;
\foreach \a in{a,...,e}
\draw[fill=black] (\a) circle [radius=1mm];
\end{tikzpicture} \ \ \ \ \ \ \ \ \ 
\begin{tikzpicture}[scale=0.5]
\coordinate (a) at (0,0) {};
\coordinate (b) at (-1,1) {};
\coordinate (c) at (-1,2) {};
\coordinate (d) at (0,3) {};
\coordinate (e) at (1,1.5) {};
\coordinate (f) at (0,4) {};
\draw(a)--(b)--(c)--(d)--(e)--cycle;
\draw(d)--(f);
\foreach \a in{a,...,f}
\draw[fill=black] (\a) circle [radius=1mm];
\end{tikzpicture} \ \ \ \ \ \ \ \ \ 
\begin{tikzpicture}[scale=0.5]
\coordinate (a) at (0,0) {};
\coordinate (b) at (-1,1) {};
\coordinate (c) at (-1,2) {};
\coordinate (d) at (0,3) {};
\coordinate (e) at (1,1.5) {};
\coordinate (f) at (0,-1) {};
\draw(a)--(b)--(c)--(d)--(e)--cycle;
\draw(a)--(f);
\foreach \a in{a,...,f}
\draw[fill=black] (\a) circle [radius=1mm];
\end{tikzpicture} \ \ \ \ \ \ \ \ \ 
\begin{tikzpicture}[scale=0.5]
\coordinate (a) at (0,0) {};
\coordinate (b) at (-1,0.75) {};
\coordinate (c) at (-1,1.5) {};
\coordinate (d) at (-1,2.25) {};
\coordinate (e) at (0,3) {};
\coordinate (f) at (1,1.5) {};
\draw(a)--(b)--(c)--(d)--(e)--(f)--cycle;
\foreach \a in{a,...,f}
\draw[fill=black] (\a) circle [radius=1mm];
\end{tikzpicture} \ \ \ \ \ \ \ \ \ 
\begin{tikzpicture}[scale=0.5]
\coordinate (a) at (0,0) {};
\coordinate (b) at (-1,1) {};
\coordinate (c) at (-1,2) {};
\coordinate (d) at (0,3) {};
\coordinate (e) at (1,2) {};
\coordinate (f) at (1,1) {};
\draw(a)--(b)--(c)--(d)--(e)--(f)--cycle;
\foreach \a in{a,...,f}
\draw[fill=black] (\a) circle [radius=1mm];
\end{tikzpicture}
\caption{Lattices with $\reg(I_L)=3$.}
\end{figure}
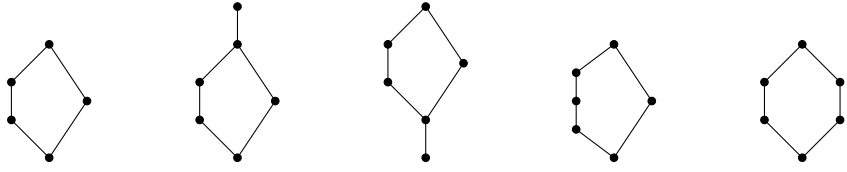

\begin{Conjecture}
    {\em
The join--meet ideal of a nonmodular lattice is not linearly related.   
    }
\end{Conjecture}

Let $2 \leq n \leq m $ and $D_{2^n \cdot 3^m}$ the divisor lattice of $2^n \cdot 3^m$.  Let $L_{n,m}$ denote the meet-semilattice (\cite[p.~286]{ECI}) consisitng of those $2^a\cdot 3^b$ with $a+b \leq m$.  Set ${L_{n,m}^{\sharp}} := L_{n,m} \cup \{{\hat 1}\}$, where ${\hat 1} \not\in L_{n,m}$ is the unique maximal element of ${L_{n,m}^{\sharp}}$.  It follows from \cite[Proposition 3.3.1]{ECI} that ${L_{n,m}^\sharp}$ is a lattice.  In fact, ${L_{n,m}^{\sharp}}$ is a nonmodular lattice.  For example, ${L_{2,2}^{\sharp}}$ is the nonmodular lattice of Figure $8$.

\begin{figure}
\centering
\begin{tikzpicture}[scale=0.7]
\coordinate (a) at (0,0) {};
\coordinate (b) at (1,-1) {};
\coordinate (c) at (2,-2) {};
\coordinate (d) at (3,-3) {};
\coordinate (e) at (4,-2) {};
\coordinate (f) at (3,-1) {};
\coordinate (g) at (2,0) {};
\coordinate (h) at (1,1) {};
\coordinate (i) at (2,2) {};
\coordinate (j) at (3,1) {};
\coordinate (k) at (4,0) {};
\coordinate (l) at (5,-1) {};
\coordinate (m) at (6,0) {};
\coordinate (n) at (5,1) {};
\coordinate (o) at (4,2) {};
\coordinate (p) at (5,3) {};
\coordinate (q) at (6,2) {};
\coordinate (r) at (7,1) {};
\coordinate (s) at (8,2) {};
\draw(a)--(b)--(c)--(d)--(e)--(f)--(g)--(h)--(i)--(j)--(k)--(l)--(m)--(n)--(o)--(p)--(q)--(r)--(s)--(p)--(i);
\draw(a)--(h);
\draw(b)--(g)--(j)--(o);
\draw(c)--(f)--(k)--(n)--(q);
\draw(e)--(l);
\draw(m)--(r);
\draw[fill=white] (a) circle [radius=2mm];
\draw[fill=white] (b) circle [radius=2mm];
\draw[fill=white] (c) circle [radius=2mm];
\draw[fill=white] (d) circle [radius=2mm];
\draw[fill=white] (e) circle [radius=2mm];
\draw[fill=white] (f) circle [radius=2mm];
\draw[fill=white] (g) circle [radius=2mm];
\draw[fill=white] (h) circle [radius=2mm];
\draw[fill=white] (i) circle [radius=2mm];
\draw[fill=white] (j) circle [radius=2mm];
\draw[fill=white] (k) circle [radius=2mm];
\draw[fill=white] (l) circle [radius=2mm];
\draw[fill=white] (m) circle [radius=2mm];
\draw[fill=white] (n) circle [radius=2mm];
\draw[fill=white] (o) circle [radius=2mm];
\draw[fill=white] (p) circle [radius=2mm];
\draw[fill=white] (q) circle [radius=2mm];
\draw[fill=white] (r) circle [radius=2mm];
\draw[fill=white] (s) circle [radius=2mm];
\end{tikzpicture}
\caption{The nonomodular lattice ${L_{3,5}^{\sharp}}$.}
\end{figure}
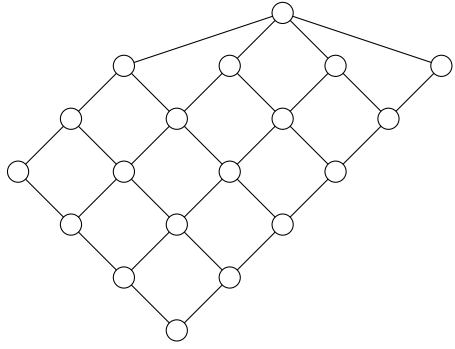

\begin{Theorem}
\label{Lnmsharp}
Let $2 \leq n \leq m $.  The join-meet ideal of the nonmodular lattice ${L_{n,m}^\sharp}$ is not linearly related.  
\end{Theorem}

\begin{proof}
The interval $[3^{m-2}, {\hat 1}]$ of ${L_{n,m}^\sharp}$ is the nonmodular lattice of Figure $7$.  One easily see that $[3^{m-2}, {\hat 1}]$ is an {\em induced} sublattice (\cite[p.~485]{DEH}) of ${L_{n,m}^\sharp}$.  Since the join--meet ideal of $[3^{m-2}, {\hat 1}]$ is not linearly related (Example \ref{aaaaa}), it follows from \cite[Proposition 2.1]{DEH} that the join--meet ideal of ${L_{n,m}^\sharp}$ is not linearly related. \, \, \, \, 
\end{proof}

We now turn to the discussion of finite pure lattices of rank $3$  (\cite[pp.~115--116]{Hibired}) with at least seven elements.  Let $L$ be a finite pure lattice of rank $3$ with $|L| \geq 7$ and ${\hat 0}$ (resp. ${\hat 1}$) its unique minimal (resp. maximal) element.  We say that $L$ is {\em connected} if the bipartite graph $G_L := L \setminus \{{\hat 0}, {\hat 1}\}$ is connected.


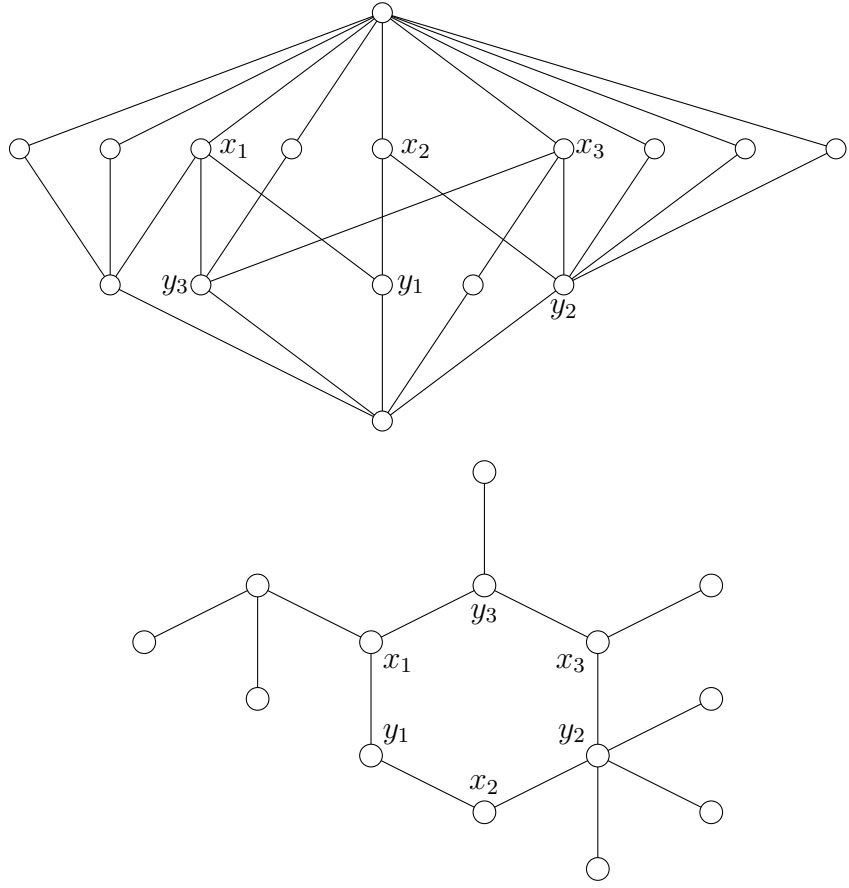
\begin{figure}
\centering
\begin{tikzpicture}[scale=1.2]
\coordinate (a) at (0,0) {};
\coordinate (b) at (-3,1.5) {};
\coordinate (c) at (-2,1.5) {};
\coordinate (d) at (0,1.5) {};
\coordinate (e) at (1,1.5) {};
\coordinate (f) at (2,1.5) {};
\coordinate (g) at (-4,3) {};
\coordinate (h) at (-3,3) {};
\coordinate (i) at (-2,3) {};
\coordinate (j) at (-1,3) {};
\coordinate (k) at (0,3) {};
\coordinate (l) at (2,3) {};
\coordinate (m) at (3,3) {};
\coordinate (n) at (4,3) {};
\coordinate (o) at (5,3) {};
\coordinate (p) at (0,4.5) {};
\foreach \k in{b,...,f}
\draw(a)--(\k);
\foreach \a in{g,...,i}
\draw(b)--(\a);
\foreach \b in{i,j,l}
\draw(c)--(\b);
\foreach \c in{i,k}
\draw(d)--(\c);
\draw(e)--(l);
\foreach \e in{k,...,o}
\draw(f)--(\e);
\foreach \f in{g,...,o}
\draw(p)--(\f);
\foreach \g in{a,...,p}
\draw[fill=white] (\g) circle [radius=1.1mm];
\draw(i)node[right=1mm]{$x_1$};
\draw(k)node[right=1mm]{$x_2$};
\draw(l)node[right=0.1mm]{$x_3$};
\draw(c)node[left=0.1mm]
{$y_3$};
\draw(d)node[right=0.5mm]{$y_1$};
\draw(f)node[below=0.5mm]{$y_2$};
\end{tikzpicture}
\\ 
\bigskip
\begin{tikzpicture}[scale=1.5]
\coordinate (a) at (0,0) {};
\coordinate (b) at (-1,0.5) {};
\coordinate (c) at (-1,1.5) {};
\coordinate (d) at (0,2) {};
\coordinate (e) at (1,1.5) {};
\coordinate (f) at (1,0.5) {};
\coordinate (g) at (2,2) {};
\coordinate (h) at (2,1) {};
\coordinate (i) at (2,0) {};
\coordinate (j) at (1,-0.5) {};
\coordinate (k) at (-2,2) {};
\coordinate (l) at (-2,1) {};
\coordinate (m) at (-3,1.5) {};
\coordinate (n) at (0,3) {};
\draw(a)--(b)--(c)--(d)--(e)--(f)--cycle;
\draw(e)--(g);
\draw(f)--(h);
\draw(f)--(i);
\draw(f)--(j);
\draw(c)--(k)--(m);
\draw(k)--(l);
\draw(d)--(n);
\foreach \k in{a,...,n}
\draw[fill=white] (\k) circle [radius=1mm];
\draw(a)node[above=1mm]{$x_2$};
\draw(b)node[above right=0.1mm]{$y_1$};
\draw(c)node[below right=0.1mm]{$x_1$};
\draw(d)node[below=1mm]{$y_3$};
\draw(e)node[below left=0.1mm]{$x_3$};
\draw(f)node[above left=0.1mm]{$y_2$};
\end{tikzpicture}
\caption{A finite connected pure lattice of rank $3$ and its bipartite graph.}
\end{figure}

Let $L$ be a finite connected pure lattice of rank $3$ with at least seven elements and $G_L$ its bipartite graph.  Let $[n]$ be the vertex set of $G_L$ and $$S=K[x_a : a \in L]=K[x_1, \ldots, x_n, x_{\hat 0}, x_{\hat 1}]$$ the polynomial ring in $n+2$ variables over a field $K$. Let $\preceq$ be the reverse lexicographic order on $S$ induced by the ordering $x_1\geq \cdots \geq x_n\geq x_{\hat 0}\geq x_{\hat 1}$  (\cite[Example 2.1.2 (b)]{HHgtm260}).

\begin{Lemma} \label{GBcycle}
Suppose that $G_L$ is a cycle $C_{2n}$ of length $2n\geq 8$ with vertex set $[2n]=\{1,2,\ldots, 2n\}$ and with edge set$$E(C_{2n})=\big\{\{1,2\},\{2,3\}, \ldots, \{2n-1,2n\},\{1,2n\}\big\}.$$ Then
\begin{align*}
\mathcal{G} & =\{x_ax_b-x_{a\wedge b}x_{a\vee b} : a,b \ {\rm are \ incomparable \ in} \ L\}\\ & \, \, \, \, \, \, \, \, \, \, \, \, \, \, \, \, \, \, \bigcup \, \{x_{\hat 0}x_ax_{\hat 1}-x_{\hat 0}x_{\hat 1}^2 : a\in L\setminus \{{\hat 1}\}\} 
\end{align*}
is the reduced Gr\"obner bases of the join-meet ideal $I_{L}$ of $L$ with respect to $\preceq$.
\end{Lemma}

\begin{proof}
To simplify the notation, set $I:=I_L$. Let $U=\{1, 3, \ldots, 2n-1\}$ and $V=\{2,4,\dots, 2n\}$. Without loss of generality, we may assume that where ${\hat 0} < a < b < {\hat 1}$, for any $a\in U$ and $b\in V$ with 
$\{a,b\}\in E(G_L)$.
First we show that $\mathcal{G}$ is contained in $I$. It is enough to prove that $x_{\hat 0}x_ax_{\hat 1}-x_{\hat 0}x_{\hat 1}^2\in I$ for each $a\in L\setminus \{{\hat 1}\}$. Let $a\neq {\hat 0}$. Then $a\in [2n]$. If $a\in V$, then by symmetry, one may assume that $a=2$. One has 
\begin{align*}
& \, x_{2n-1}f_{2,4}-x_2f_{4,2n-1}+x_{\hat 1}f_{3,2n-1}\\ = & \, x_{2n-1}(x_2x_4-x_3x_{\hat 1})-x_2(x_4x_{2n-1}-x_{\hat 0}x_{\hat 1})+x_{\hat 1}(x_3x_{2n-1}-x_{\hat 0}x_{\hat 1})\\ = & \, x_{\hat 0}x_2x_{\hat 1}-x_{\hat 0}x_{\hat 1}^2.    
\end{align*}
Hence, $x_{\hat 0}x_2x_{\hat 1}-x_{\hat 0}x_{\hat 1}^2\in I$. If $a\in U$, then one may assume that $a=1$. One has 
\begin{align*}
& x_2f_{1,6}-x_1f_{2,6}=x_2(x_1x_6-x_{\hat 0}x_{\hat 1})-x_1(x_2x_6-x_{\hat 0}x_{\hat 1})=x_{\hat 0}x_1x_{\hat 1}-x_{\hat 0}x_2x_{\hat 1}   
\end{align*}
Since we have already shown that $x_{\hat 0}x_2x_{\hat 1}-x_{\hat 0}x_{\hat 1}^2\in I$, the above equalities imply that $x_{\hat 0}x_1x_{\hat 1}-x_{\hat 0}x_{\hat 1}^2\in I$. Let $a={\hat 0}$. Then
\begin{align*}
& \, x_1f_{4,2n-1}-x_4f_{1,2n-1}-x_{\hat 0}f_{4,2n}\\ = & \, x_1(x_4x_{2n-1}-x_{\hat 0}x_{\hat 1})-x_4(x_1x_{2n-1}-x_{\hat 0}x_{2n})-x_{\hat 0}(x_4x_{2n}-x_{\hat 0}x_{\hat 1})\\ = & \, x_{\hat 0}^2x_{\hat 1}-x_{\hat 0}x_1x_{\hat 1}. 
\end{align*} 
Since, $x_{\hat 0}x_1x_{\hat 1}-x_{\hat 0}x_{\hat 1}^2\in I$, we conclude from the above equalities that $x_{\hat 0}^2x_{\hat 1}-x_{\hat 0}x_{\hat 1}^2\in I$. 
As a consequence, $\mathcal{G}\subseteq I$.

To prove that $\mathcal{G}$ is a Gr\"obner bases of $I$ with respect to $\preceq$, we must show that for any pair of polynomials $g_1, g_2\in \mathcal{G}$,  the $S$-polynomial $S(g_1,g_2)$ reduces to zero with respect to $\mathcal{G}$ (\cite[Theorem 2.3.2]{HHgtm260}). We divide the following cases.

\medskip

\noindent
{\bf (Case 1.)} Let $g_1=f_{a_1,b_1}$ and $g_2=f_{a_2,b_2}$, where $a_1,b_1,a_2,b_2\in U$. If $\{a_1,b_1\}\cap\{a_2,b_2\}=\emptyset$, then ${\rm in}_{\preceq}(g_1)$ and ${\rm in}_{\preceq}(g_2)$ are relatively prime. Thus, $S(g_1,g_2)$ reduces to zero with respect to $\mathcal{G}$. Let $\{a_1,b_1\}\cap\{a_2,b_2\}\neq\emptyset$. One may assume that $a_1=a_2:=a$. By symmetry, we assume that $a=1$. Furthermore, one may assume that $b_1< b_2$. The following subcases arise.

{\bf Subcase 1.1.} If $b_1=3$ and $b_2=2n-1$, then $g_1=x_1x_3-x_{\hat 0}x_2$ and $g_2=x_1x_{2n-1}-x_{\hat 0}x_{2n}$. Hence,
\begin{align*}
S(g_1,g_2)=& \, x_{\hat 0}x_3x_{2n}-x_{\hat 0}x_2x_{2n-1}\\ =& \,x_{\hat 0}(x_3x_{2n}-x_{\hat 0}x_{\hat 1})-x_{\hat 0}(x_2x_{2n-1}-x_{\hat 0}x_{\hat 1})\\ = & \, x_{\hat 0}f_{3,2n}-x_{\hat 0}f_{2,2n-1}  
\end{align*}
reduces to zero with respect to $\mathcal{G}$.

{\bf Subcase 1.2.} If $b_1=3$ and $b_2\in U\setminus \{1, 3, 2n-1\}$, then $g_1=x_1x_3-x_{\hat 0}x_2$ and $g_2=x_1x_{b_2}-x_{\hat 0}x_{\hat 1}$. Hence,
\begin{align*}
S(g_1,g_2)=& \, x_{\hat 0}x_3x_{\hat 1}-x_{\hat 0}x_2x_{b_2}\\ =&\,(x_{\hat 0}x_3x_{\hat 1}-x_{\hat 0}x_{\hat 1}^2)-(x_{\hat 0}^2x_{\hat 1}-x_{\hat0}x_{\hat 1}^2)-x_{\hat 0}(x_2x_{b_2}-x_{\hat 0}x_{\hat 1})\\ =&\,(x_{\hat 0}x_3x_{\hat 1}-x_{\hat 0}x_{\hat 1}^2)-(x_{\hat 0}^2x_{\hat 1}-x_{\hat0}x_{\hat 1}^2)-x_{\hat 0}f_{2,b_2}
\end{align*}
reduces to zero with respect to $\mathcal{G}$.

{\bf Subcase 1.3.} If $b_2=2n-1$ and $b_1\in U\setminus \{1, 3, 2n-1\}$, then the argument is similar to that of Subcase 1.2.

{\bf Subcase 1.4.} If $b_1,b_2\in U\setminus \{1, 3, 2n-1\}$, then $g_1=x_1x_{b_1}-x_{\hat 0}x_{\hat 1}$ and $g_2=x_1x_{b_2}-x_{\hat 0}x_{\hat 1}$. Hence, 
\begin{align*}
& S(g_1,g_2)=x_{\hat 0}x_{b_1}x_{\hat 1}-x_{\hat 0}x_{b_2}x_{\hat 1}=(x_{\hat 0}x_{b_1}x_{\hat 1}-x_{\hat 0}x_{\hat 1}^2)-(x_{\hat 0}x_{b_2}x_{\hat 1}-x_{\hat 0}x_{\hat 1}^2)
\end{align*}
reduces to zero with respect to $\mathcal{G}$.

\medskip
\noindent
{\bf (Case 2.)} Let $g_1=f_{a_1,b_1}$ and $g_2=f_{a_2,b_2}$, where $a_1,b_1,a_2,b_2\in V$. A similar argument as in (Case 1) shows that $S(g_1,g_2)$ reduces to zero with respect to $\mathcal{G}$.

\medskip

\noindent
{\bf (Case 3.)} Let $g_1=f_{a_1,b_1}$ and $g_2=f_{a_2,b_2}$, where $a_1,b_1\in U$ and $a_2, b_2\in V$. Since ${\rm in}_{\preceq}(g_1)$ and ${\rm in}_{\preceq}(g_2)$ are relatively prime, $S(g_1,g_2)$ reduces to zero with respect to $\mathcal{G}$.

\medskip
\noindent
{\bf (Case 4.)} Let $g_1=f_{a_1,b_1}$ and $g_2=f_{a_2,b_2}$, where $a_1,a_2\in U$ and $b_1, b_2\in V$. If $a_1\neq a_2$ and $b_1\neq b_2$, then ${\rm in}_{\preceq}(g_1)$ and ${\rm in}_{\preceq}(g_2)$ are relatively prime. 
So, suppose that either $a_1=a_2$ or $b_1=b_2$. Let $a_1=a_2:=a$, as the argument in the other case is similar. One may assume that $a=1$ and $b_1< b_2$. Since $a$ is not comparable with $b_1,b_2$, one has $b_1,b_2\notin \{2, 2n\}$. Then $g_1=x_1x_{b_1}-x_{\hat 0}x_{\hat 1}$ and $g_2=x_1x_{b_2}-x_{\hat 0}x_{\hat 1}$. Hence,
\begin{align*}
S(g_1,g_2)=x_{b_1}x_{\hat 0}x_{\hat 1}-x_{b_2}x_{\hat 0}x_{\hat 1}=(x_{b_1}x_{\hat 0}x_{\hat 1}-x_{\hat 0}x_{\hat 1}^2)-(x_{b_2}x_{\hat 0}x_{\hat 1}-x_{\hat 0}x_{\hat 1}^2)   
\end{align*}
reduces to zero with respect to $\mathcal{G}$.

\medskip
\noindent
{\bf (Case 5.)} Let $g_1=f_{a_1,b_1}$ and $g_2=f_{a_2,b_2}$, where $a_1,a_2,b_1\in U$ and $b_2\in V$. If $a_1\neq a_2$, then ${\rm in}_{\preceq}(g_1)$ and ${\rm in}_{\preceq}(g_2)$ are relatively prime. 
So, suppose that $a_1=a_2:=a$. One may assume that $a=1$. If $1\vee b_1={\hat 1}$, then the argument is the same as in (Case 4). Let $1\vee b_1\neq {\hat 1}$. In other words, either $b_1=3$ or $b_1=2n-1$. By symmetry, one ma assume that $b_1=3$. One has $g_1=x_1x_3-x_2x_{\hat 0}$ and $g_2=x_1x_{b_2}-x_{\hat 0}x_{\hat 1}$. Since $1$ and $b_2$ are incomparable in $L$, one has $b_2\neq 2, 2n$. If $b_2=4$, then
\begin{align*}
 S(g_1,g_2)=x_3x_{\hat 0}x_{\hat 1}-x_2x_{b_2}x_{\hat 0}=x_3x_{\hat 0}x_{\hat 1}-x_2x_4x_{\hat 0}=-x_{\hat 0}(x_2x_4-x_3x_{\hat 1})=-x_{\hat 0}f_{2,4} 
\end{align*}
reduces to zero with respect to $\mathcal{G}$. If $b_2\neq 4$, then
\begin{align*}
S(g_1,g_2) =& \, x_3x_{\hat 0}x_{\hat 1}-x_2x_{b_2}x_{\hat 0}\\ = &\, (x_3x_{\hat 0}x_{\hat 1}-x_{\hat 0}x_{\hat 1}^2)-(x_{\hat 0}^2x_{\hat 1}-x_{\hat 0}x_{\hat 1}^2)-x_{\hat 0}(x_2x_{b_2}-x_{\hat 0}x_{\hat 1}) 
\end{align*}
reduces to zero with respect to $\mathcal{G}$.

\medskip
\noindent
{\bf (Case 6.)} Let $g_1=f_{a_1,b_1}$ and $g_2=f_{a_2,b_2}$, where $a_1,a_2,b_1\in V$ and $b_2\in U$. A similar argument as in (Case 5) shows that $S(g_1,g_2)$ reduces to zero with respect to $\mathcal{G}$.

\medskip
\noindent
{\bf (Case 7.)} Let $g_1=f_{a_1,b_1}$ and $g_2=x_{\hat 0}x_ax_{\hat 1}-x_{\hat 0}x_{\hat 1}^2$. If $a\notin\{a_1,b_1\}$, then ${\rm in}_{\preceq}(g_1)$ and ${\rm in}_{\preceq}(g_2)$ are relatively prime. 
Let $a\in\{a_1,b_1\}$ and, say, $a=a_1$.  Let, say, $a\in U$ and, by symmetry, suppose that $a=1$. If $1\vee b_1={\hat 1}$, then $g_1=x_1x_{b_1}-x_{\hat 0}x_{\hat 1}$. Hence,
\begin{align*}
S(g_1,g_2)=x_{b_1}x_{\hat 0}x_{\hat 1}^2-x_{\hat 0}^2x_{\hat 1}^2=x_{\hat 1}(x_{b_1}x_{\hat 0}x_{\hat 1}-x_{\hat 0}x_{\hat 1}^2)-x_{\hat 1}(x_{\hat 0}^2x_{\hat 1}-x_{\hat 0}x_{\hat 1}^2)
\end{align*}
reduces to zero with respect to $\mathcal{G}$. If $1\vee b_1\neq {\hat 1}$, then either $b_1=3$ or $b_1=2n-1$. By symmetry suppose that $b_1=3$. Then $g_1=x_1x_3-x_2x_{\hat 0}$ and
\begin{align*}
S(g_1,g_2)=&\,x_3x_{\hat 0}x_{\hat 1}^2-x_2x_{\hat 0}^2x_{\hat 1}\\ = & \, x_{\hat 1}(x_3x_{\hat 0}x_{\hat 1}-x_{\hat 0}x_{\hat 1}^2)-x_{\hat 1}(x_{\hat 0}^2x_{\hat 1}-x_{\hat 0}x_{\hat 1}^2)-x_{\hat 0}(x_2x_{\hat 0}x_{\hat 1}-x_{\hat 0}x_{\hat 1}^2)
\end{align*}
reduces to zero with respect to $\mathcal{G}$.

\medskip

\noindent
{\bf (Case 8.)} Let $g_1=x_{\hat 0}x_ax_{\hat 1}-x_{\hat 0}x_{\hat 1}^2$ and $g_2=x_{\hat 0}x_bx_{\hat 1}-x_{\hat 0}x_{\hat 1}^2$, where $a,b\in L\setminus\{{\hat 1}\}$. Then 
\begin{align*}
& S(g_1,g_2)=x_ax_{\hat 0}x_{\hat 1}^2-x_bx_{\hat 0}x_{\hat 1}^2=x_{\hat 1}(x_ax_{\hat 0}x_{\hat 1}-x_{\hat 0}x_{\hat 1}^2)-x_{\hat 1}(x_bx_{\hat 0}x_{\hat 1}-x_{\hat 0}x_{\hat 1}^2)
\end{align*}
reduces to zero with respect to $\mathcal{G}$.
\hspace{7.4cm}
\end{proof}

\begin{Lemma} \label{GBadd}
Working in the same situation as in Lemma \ref{GBcycle},
\begin{align*}
\mathcal{G} & =\{x_ax_b-x_{a\wedge b}x_{a\vee b} : a,b \ {\rm are \ incomparable \ in} \ L\} \bigcup \{x_{\hat 0}x_{\hat 1}\} 
\end{align*}
is a Gr\"obner bases of $(I_{L}, x_{\hat 0}x_{\hat 1})$ with respect to $\preceq$.   
\end{Lemma}

\begin{proof}
Since $\mathcal{G}\subseteq (I_{L_G}, x_{\hat 0}x_{\hat 1})$, we show that for any $g_1,g_2\in \mathcal{G}$, the S-polynomial $S(g_1,g_2)$ reduces to zero with respect to $\mathcal{G}$. If one of $g_1$ and $g_2$ is $x_{\hat 0}x_{\hat 1}$, then ${\rm in}_{\preceq}(g_1)$ and ${\rm in}_{\preceq}(g_2)$ are relatively prime. Thus $S(g_1,g_2)$ reduces to zero with respect to $\mathcal{G}$.  Let $g_1=f_{a_1,b_1}$ and $f_{a_2,b_2}$, where $a_1,a_2,b_1,b_2\in [2n]$. It follows from Lemma \ref{GBcycle} that $S(g_1,g_2)$ reduces to zero with respect to  
\begin{align*}
\mathcal{G}' & =\{x_ax_b-x_{a\wedge b}x_{a\vee b} : a,b \ {\rm are \ incomparable \ in} \ L\}\\ & \, \, \, \, \, \, \, \, \, \, \, \, \, \, \, \bigcup \, \{x_{\hat 0}x_ax_{\hat 1}-x_{\hat 0}x_{\hat 1}^2 : a\in L\setminus \{{\hat 1}\}\}.
\end{align*}
Since any polynomial in $\{x_{\hat 0}x_ax_{\hat 1}-x_{\hat 0}x_{\hat 1}^2 : a\in L_G\setminus \{{\hat 1}\}\}$ is divisible by $x_{\hat 0}x_{\hat 1}$, it follows that $S(g_1,g_2)$ reduces to zero with respect to $\mathcal{G}$.
\hspace{5cm}
\end{proof}

\begin{Lemma}
\label{CCCCC}
Working in the same situation as in Lemma \ref{GBcycle}, one has
\begin{itemize}
    \item[(i)] $\depth({S/I_{L}})=1$;
    \item[(ii)] $\reg(I_{L})=4$.
\end{itemize}
\end{Lemma}

\begin{proof}
Set $I:=I_{L}$.

(i) It follows from Lemma \ref{GBcycle} that
$$(x_a-x_{\hat 1}: a\in L_G\setminus\{\hat 1\})\subseteq (I: x_{\hat 1}x_{\hat 0}).$$We show that the reverse inclusion holds as well.  Let $f\in S$ be a polynomial which belongs to $(I: x_{\hat 1}x_{\hat 0})$. Reducing modulo $(x_a-x_{\hat 1}: a\in L_G\setminus\{\hat 1\})$, one can assume that there is an integer $d\geq 0$ for which $f=c_dx_{\hat 1}^d+\cdots + c_1x_{\hat 1}+c_0$, where $c_d, \ldots, c_1,c_0\in K$. Since $fx_{\hat 0}x_{\hat 1}\in I$, we conclude that$$c_dx_{\hat 0}x_{\hat 1}^{d+1}={\rm in}_{\preceq}(fx_{\hat 0}x_{\hat 1})\in {\rm in}_{\preceq}(I),$$contradicting Lemma \ref{GBcycle}. Hence,$$(I: x_{\hat 1}x_{\hat 0})=(x_a-x_{\hat 1}: a\in L_G\setminus\{\hat 1\}),$$
which shows that $\mathfrak{p}:=(x_a-x_{\hat 1}: a\in L_G\setminus\{\hat 1\})$ is an associated prime of $I$. Consequently, one has $\depth {S/I}\leq \dim S/\mathfrak{p}=1$. To prove the reverse inequality, we consider the short exact sequence
\[
\begin{array}{rl}
\displaystyle0\longrightarrow \frac{S}{(I: x_{\hat 1}x_{\hat 0})}(-2)\longrightarrow \frac{\,S\, }{I}\longrightarrow \frac{S}{(I,x_{\hat 0}x_{\hat 1})}\longrightarrow 0.
\end{array} \tag{5} \label{5}
\] 
It yields that$$\depth({S/I})\geq\min\big\{\depth{S/\mathfrak{p}},\depth({S/(I,x_{\hat 0}x_{\hat 1})})\big\}.$$One has $\depth({S/\mathfrak{p}})=1$.
Lemma \ref{GBadd} says that $(I,x_{\hat 0}x_{\hat 1})$ has a squarefree quadratic initial ideal. In fact, this initial ideal can be written as $J+(x_{\hat 0}x_{\hat 1})$, where $J$ is the Stanley--Reisner ideal of $C_{2n}$ (considered as a $1$-dimensional simplicial complex). Hence, $\depth{S/J+(x_{\hat 0}x_{\hat 1}})=2+1=3$.  We deduce from \cite[Corollary 2.7]{CV} that $\depth{S/(I,x_{\hat 0}x_{\hat 1})}=3$. Therefore, $\depth{S/I}\geq 1$. Since we already know that $\depth{S/I}\leq 1$, the assertion follows.

(ii) As the regularity of the Stanley--Reisner ring of $C_{2n}$ is two, a similar argument as in the proof of (i) shows that $\reg{S/(I,x_{\hat 0}x_{\hat 1})}=3$. Since $\reg{S/\mathfrak{p}}=0$, we conclude from the short exact sequence (\ref{5}) that$$\reg({S/I})=\max\big\{\reg{S/\mathfrak{p}}+2,\reg{S/(I,x_{\hat 0}x_{\hat 1})}\big\}=3,$$as desired.
\hspace{12cm}
\end{proof}

\begin{Lemma} \label{GBpath}
Let $L$ be a connected pure lattice of rank $3$ and suppose that $G_L$ is a path $P_{n}$ on $n\geq 7$ vertices. Then
\begin{align*}
\mathcal{G} & =\{x_ax_b-x_{a\wedge b}x_{a\vee b} : a,b \ {\rm are \ incomparable \ in} \ L \}\\ & \, \, \, \, \, \, \, \, \, \, \, \, \, \, \, \bigcup \, \{x_{\hat 0}x_ax_{\hat 1}-x_{\hat 0}x_{\hat 1}^2 : a\in L \setminus \{{\hat 1}\}\} 
\end{align*}
is the reduced Gr\"obner bases of $I_{L}$ with respect to $\preceq$.
\end{Lemma}

\begin{proof}
The argument is similar to that of Lemma \ref{GBcycle}. So, we omit the proof. 
\, 
\end{proof}

\begin{Lemma}
\label{TTTTT}
Working in the same situation as in Lemma \ref{GBpath}, one has
\begin{itemize}
    \item[(i)] $\depth({S/I_L})=1$;
    \item[(ii)] $\reg({I_L})=3$.
\end{itemize}
\end{Lemma}

\begin{proof}
 The argument is similar to that of Lemma \ref{CCCCC}. So, we omit the proof. The difference is that one uses Lemma \ref{GBpath} instead of Lemma \ref{GBcycle}. Furthermore, in this case, the ideal $(I_L, x_{\hat 0}x_{\hat 1})$ has an initial ideal of the form $J+(x_{\hat 0}x_{\hat 1})$, where $J$ is the Stanley--Reisner ideal of $P_n$ as a $1$-dimensional simplicial complex. So, $\reg({S/J})=1$. Consequently, one has $\reg({S/J+(x_{\hat 0}x_{\hat 1}}))=2$.
 \, \, \, \, \,  \, \, \, \, \, \, \,
\end{proof}

\begin{Example}
\label{EX2.11}
    {\em
    Let $L$ be a connected pure lattice of rank $3$.  
    If $G_L$ is the path $P_{6}$, then $\depth({S/I_{L}})=2$ and  $\reg({I_{L}})=3$.
    If $G_L$ is the path $P_{5}$, then $\depth({S/I_L})=3$ and  $\reg({I_L})=4$ (Example \ref{aaaaa}).   
    }
\end{Example}

Recall that the degree of a vertex $x$ of a finite graph $G$ is the number of vertices $y$ for which $\{x,y\}$ is an edge of $G$.

\begin{Lemma}
\label{QQQQQ}
Let $L$ be a finite connected pure lattice of rank $3$ and suppose that the bipartite graph $G_L$ possesses a vertex whose degree is at least $3$. 
    Then the join-meet ideal $I_L$ is not linearly related and $\reg(I_L) \geq 4$.
\end{Lemma}

\begin{proof}
Let $x$ be a vertex whose degree is $q \geq 3$ and $y_1, \ldots, y_q$ the vertices with each $\{x, y_i\}$ is an edge of $G$.  One can assume that ${\hat 0} < y_i < x$ in $L_G$.  Thus the interval $[{\hat 0}, x]$ is the (generalized) diamond lattice $D_{q+2}$ (\cite[p.~481]{DEH}).  Moreover, it is easy to see that $[{\hat 0},x]$ is an induced sublattice of $L$.  Since the join-meet ideal of $D_{n+2}$ is not linearly related and its regularity is at least $4$ (\cite[Theorem 2.2]{DEH}), it follows from \cite[Proposition 2.1]{DEH} that the join--meet ideal of $L$ is not linearly related with $\reg(I_L)\geq4$.
\, \, \, \, \, 
\end{proof}

\begin{Theorem}
    \label{rank3lattice}
Let $L$ be a finite connected pure lattice of rank $3$ with $|L| \geq 7$ and $G_L$ its bipartite graph. 

(1)  One has $\reg(I_L) \geq 3$.  In particular, $I_L$ does not have linear resolution.  Furthermore, $\reg(I_L) = 3$ if and only if $G_L$ is (i) the cycle  $C_6$ of length $6$ or (ii) a path $P_n$ with $n \geq 6$.

(2) Suppose that the bipartite $G_L$ is neither a cycle nor a path.     
    Then the join-meet ideal of $L$ is not linearly related.   
\end{Theorem}

\begin{proof}
If $G_L$ is neither a cycle nor a path, then the degree of one of the vertices of $G_L$ must be at least $3$.  It then follows from Lemma \ref{QQQQQ} that $I_L$ is not linearly related and $\reg(I_L) \geq 4$.  When $G_L$ is either a cycle or a path, the desired result follows from Examples \ref{EX2.2}, \ref{EX2.11}
and Lemmata \ref{CCCCC}, \ref{TTTTT}. Note that when $G$ is the cycle $C_6$, then $L$ is the boolean lattice $B_3$, and in this case $\reg(I_L)=3$.
\hspace{2.3cm}
\end{proof}

\begin{Question}
    Can one classify finite nonmodular lattices $L$ with $\reg(I_L) = 3$?
\end{Question}

\section*{Acknowledgments}
The second author is supported by a FAPA grant from Universidad de los Andes.

\section*{Statements and Declarations}
The authors have no conflict of interest to declare that are relevant to the content of this article.

\section*{Data availability}
Data sharing does not apply to this article as no new data were created or analyzed in this study.


\begin{thebibliography}{99}
\bibitem{CV}
A.~Conca and M.~Varbaro, Square-free Gröbner degenerations, {\em Inventiones Mathematicae} {\bf 221} (2020), 713--730.
\bibitem{DEH}
R.~Dinu, V.~Ene and T.~Hibi, On the regularity of join-meet ideals of modular lattices, {\em J. Commutative Alg.} {\bf 13} (2021), 479--488.
\bibitem{EHH} 
V.~Ene, J.~Herzog and T.~Hibi, Linearly related polyominoes, {\it J. Algebraic Combin.} {\bf 41} (2015), 949--968.
\bibitem{EH}
V.~Ene and T.~Hibi, The join-meet ideal of a finite lattice, {\em J. Commutative Algebra} {\bf 5} (2013), 209--230. 
\bibitem{HNTT}
T.~H.~H\`a, H.~D.~Nguyen, N.~V.~Trung, and T.~N.~Trung, Symbolic powers of sums
of ideals, {\em Math. Z.} {\bf 282} (2016), 819--838.
\bibitem{H}
J.~Herzog, A generalization of the Taylor complex construction, {\em Comm. Algebra}  {\bf 35} (2007), 1747--1756.
\bibitem{HHgtm260} 
J.~Herzog and T.~Hibi, ``Monomial Ideals'', GTM 260, Springer, 2011.
\bibitem{HH} 
J.~Herzog, T.~Hibi, Ideals generated by adjacent $2$-minors, {\em J. Commutative Algebra} {\bf 4} (2012), 525--549.
\bibitem{HHHKR}
J.~Herzog, T.~Hibi, F.~Hreinsd{\'o}ttir, T.~Kahle and J.~Rauh, Binomial edge ideals and conditional independence statements, {\em Adv. Appl. Math.} {\bf 45} (2010), 317--333.
\bibitem{HHO}
J.~Herzog, T.~Hibi, H.~Ohsugi, ``Binomial Ideals,'' GTM 279, Springer, 2017.
\bibitem{Hibi}
T.~Hibi, Distributive lattices, affine semigroup rings and algebras with 
straightening laws, {\it in} ``Commutative Algebra and Combinatorics'' 
(M. Nagata and H. Matsumura, Eds.), Advanced Studies in Pure Math., 
Volume 11, North--Holland, Amsterdam, 1987, pp. 93--109.
\bibitem{Hibired}
T.~Hibi, ``Algebraic Combinatorics on Convex Polytopes,'' Carslaw, Glebe, NSW, Australia, 1992. 
\bibitem{HNQS}
T.~Hibi, F.~Navarra, A.~A.~Qureshi and S.~Saeedi Madani, Minimal primes and radicality of ideals generated by adjacent 2-minors, arXiv: 2512.21449.
\bibitem{OHH}
H.~Ohsugi, J.~Herzog and T.~Hibi, Combinatorial pure subrings, {\em Osaka J. Math.} {\bf 37} (2000), 745--757.
\bibitem{Q}
A.~A.~Qureshi, Ideals generated by 2-minors, collections of cells and stack polyominoes, {\em J. Algebra} {\bf 357} (2012), 279--303.
\bibitem{ECI}
R.~Stanley, ``Enumerative Combinatorics, Volume $1$'' (Second Edition), Cambridge University Press, 2011.
\end{thebibliography}
\end{document}